\documentclass[12pt]{amsart}
\usepackage{amsmath}
\usepackage{amssymb}
\usepackage{amsthm}
\usepackage{array}
\usepackage{xy}
\usepackage[pdftex]{graphicx}
\usepackage{hyperref}
\usepackage{color}
\usepackage{transparent}
\usepackage{latexsym}

\usepackage{amsmath,scalefnt}
\usepackage{amsthm}
\usepackage{amssymb}
\usepackage{graphicx}
\usepackage{setspace}
\usepackage{hyperref}

\numberwithin{equation}{section}

\def\RR{\mathbb{R}}

\def\HH{\mathcal{H}}

\theoremstyle{plain}

\theoremstyle{plain}
\newtheorem{theorem}{Theorem} [section]
\newtheorem{corollary}[theorem]{Corollary}
\newtheorem{lemma}[theorem]{Lemma}
\newtheorem{proposition}[theorem]{Proposition}

\theoremstyle{definition}

\theoremstyle{remark}
\newtheorem{remark}[theorem]{Remark}

\numberwithin{theorem}{section}
\numberwithin{equation}{section}
\numberwithin{figure}{section}

\def\CC{\mathbb{C}}

\renewcommand{\Im}{\mathop{\rm Im}}

\begin{document}
\title[]{On the stability of the polygonal isoperimetric inequality}
\author[E. Indrei and L. Nurbekyan]{E. Indrei and L. Nurbekyan}

\def\signei{\bigskip\begin{center} {\sc Emanuel Indrei\par\vspace{3mm}Center for Nonlinear Analysis\\  
Carnegie Mellon University\\
Pittsburgh, PA 15213, USA\\
email:} {\tt eindrei@msri.org }
\end{center}}

\def\signln{\bigskip\begin{center} {\sc Levon Nurbekyan \par\vspace{3mm}
Center for Mathematical Analysis,\\
Geometry, and Dynamical Systems\\
Departamento de Matem\'atica\\
Instituto Superior T\'ecnico\\
Lisboa 1049-001, Portugal\\
email:} {\tt lnurbek@math.ist.utl.pt}
\end{center}}


\makeatletter
\def\blfootnote{\xdef\@thefnmark{}\@footnotetext}
\makeatother

\blfootnote{E. Indrei acknowledges support from the Australian Research Council, US NSF Grant DMS-0932078 administered by the Mathematical Sciences Research Institute in Berkeley, CA, and US NSF PIRE Grant OISE-0967140 administered by the Center for Nonlinear Analysis at Carnegie Mellon University. L. Nurbekyan acknowledges support from the department of mathematics at the University of Texas at Austin and the Center for Mathematical Analysis, Geometry, and Dynamical Systems at Instituto Superior T\'ecnico.}

\date{}

\maketitle

\begin{abstract}
We obtain a sharp lower bound on the isoperimetric deficit of a general polygon in terms of the variance of its side lengths, the variance of its radii, and its deviation from being convex. Our technique involves a functional minimization problem on a suitably constructed compact manifold and is based on the spectral theory for circulant matrices. 
\end{abstract}

\section{Introduction}

The stability problem for functional and geometric inequalities consists of identifying a suitable quantity which measures the deviation of a given set or function from a minimizer and serves as a lower bound on the deficit in the inequality. For instance, the classical isoperimetric inequality states that if $E \subset \mathbb{R}^n$ is a Borel set of finite Lebesgue measure $|\cdot|$, then $$\mathcal{P}(B)\le \mathcal{P}(E),$$ where $B$ is the ball with $|B|=|E|$ and $\mathcal{P}$ denotes the (distributional) perimeter. Moreover, equality holds if and only if $E$ is a ball. Recently, it was shown in \cite{FMP} that 
\begin{equation} \label{iso}
\alpha^2(E)\lesssim (\mathcal{P}(E)/\mathcal{P}(B))-1,
\end{equation}
where $$\alpha(E):=\min \Bigg \{ \frac{|E \Delta (x+B)|}{|E|}: x \in \RR^n \Bigg\}.$$ The right-hand side of \eqref{iso} is known in the literature as the isoperimetric deficit and measures how far a given set is from having minimal perimeter whereas the left-hand side is a measure of the asymmetry of the set (i.e. its ``closeness" to a ball). This result was obtained through symmetrization techniques and settled a conjecture of R.R. Hall \cite{Hall} (the exponent $2$ is sharp in any dimension).   

In \cite{FiMP}, the authors developed a method based on optimal transport theory to establish an analogous estimate for the anisotropic isoperimetric inequality. Mass transfer techniques were also employed in proving a quantitative version of the relative isoperimetric inequality inside convex cones \cite{FI}, and a stability inequality for the Gaussian isoperimetric inequality was established in \cite{Gauss} via symmetrization techniques. Moreover, there has been a lot of recent research activity directed towards proving quantitative versions of several other fundemental inequalities in analysis such as the Sobolev \cite{Sob}, log-Sobolev \cite{logsob, ls2, logsob2}, and Brunn-Minkowski \cite{B, B2} inequalities.  

In this paper, we establish a sharp stability result for the polygonal isoperimetric inequality by  introducing a method based on circulant matrix theory. It is well known that the convex regular polygon uniquely minimizes the perimeter among all polygons subject to an area constraint. In other words, if $L_*$ denotes the perimeter of the convex regular $n$-gon with area $F$, then $L_* \le L(P) $ for any $n$-gon $P$ with area $F$ and equality holds  if and only if $P$ is convex and regular. Since $L_*^2=4n \tan \frac{\pi}{n}F$, an equivalent formulation is that for any polygon $P$, $$4n \tan \frac{\pi}{n}F(P)\le L^2(P),$$ with equality if and only if $P$ is convex and regular.

Stability results for the polygonal isoperimetric inequality have been investigated in the literature by several authors. For instance, a quantitative hexagonal isoperimetric inequality appeared in Hales' proof of the celebrated Honeycomb conjecture \cite[Theorem 4]{Hales}. Moreover, Zhang \cite[Theorem 3.1]{Z} used differential inequalities involving Schur functions to obtain the following Bonnesen-type \footnote{Quantitative isoperimetric inequalities are known in the literature as Bonnesen-type inequalities.} inequality for cyclic polygons,  
$$(L(P)-L_*)^2 \le \delta(P),$$ 
where $\delta(P):=L^2(P)-4n \tan \frac{\pi}{n}F(P)$ is the polygonal isoperimetric deficit (see also \cite{Z2,Z3}). A version of this inequality also holds for general $n$-gons by scaling $L_*$ appropriately. However, if $\delta(P)=0$, then one would still need to rely on additional arguments to conclude that $P$ is convex and regular. On the other hand, Fisher, Ruoff and Shilleto \cite[Theorem 4.4]{shilleto'85} introduced a notion of orthogonal polygons and proved a general stability inequality from which it follows that if $P$ is an equilateral $n$-gon, then
\begin{equation}\label{eq: shilleto}
\sigma_r^2(P) \lesssim \delta(P),
\end{equation}
where $\sigma_r^2(P)$ is the variance of the radii of $P$ (i.e. the distances between the vertices and the barycenter). This estimate implies that if $P$ is equilateral and $\delta(P)=0$, then $P$ is cyclic, and in particular, regular. However, this by itself is not sufficient to conclude that $P$ is the minimizer since there exist regular star-shaped polygons (e.g. the pentagram); moreover, the assumption that $P$ is equilateral is restrictive. 

In the same way that $\sigma_r^2(P)$ measures the deviation of $P$ from being cyclic, the variance of the side lengths of $P$, denoted by $\sigma_s^2(P)$, measures how far $P$ is from being equilateral. With this in mind, we define the \textit{variation} of a polygon as $$v(P):=\sigma_s ^2(P) + \sigma_r^2(P).$$  As noted above, the variation by itself is not enough to identify minimizers of the polygonal isoperimetric inequality since it may fail to detect convexity; in order to account for this, we utilize a generalization of the Erd{\H o}s-Nagy theorem which states that a polygon may be convexified in a finite number of ``flips" while keeping the perimeter invariant. To illustrate the concept of a flip, consider the convex hull of a simple $n$-gon $P$. If there are pockets (i.e. a maximal connected region exterior to the polygon and interior to the convex hull), reflect one pocket across its line of support to obtain a new simple $n$-gon with the same perimeter but greater area. Then the process is repeated and it turns out that after finitely many flips, the end result is a convex polygon. This theorem is well-known although several incorrect proofs have appeared in the literature; we refer the reader to \cite[Table 1]{ENN} for a list. The generalization of this result to non-simple (i.e. self-crossing) polygons was carried out by several authors but requires a sequence of well-chosen flips to avoid infinite flipping cycles, see e.g. \cite[\S 4.1]{ENN}. In particular, Toussaint \cite{EN} uses the result for simple polygons as a black box and constructs a flip sequence which requires $O(n)$ time to determine the next flip. With this in mind, consider $$\tau(P):=\sum_{i=1}^{k_n} \tau_i(P),$$ where $\tau_i(P)\ge 0$ is the area increase at the $i$-th step given by Toussaint's process of selecting flips. An important feature of $\tau$ is that it identifies convexity: $\tau(P)=0$ if and only if $P$ is convex. Our main result consists of the following Bonnesen-type isoperimetric inequality for general polygons.

\begin{theorem}\label{thm: main}
Let $n \ge 3$ and $P$ be an $n$-gon. There exists $C(n)>0$ such that  
\begin{equation}\label{eq: goal} 
\tau(P)+v(P) \le C(n) \delta(P).
\end{equation}
\end{theorem}

\noindent Note that if $\delta(P)=0$, then Theorem \ref{thm: main} immediately implies that $P$ is convex and regular. Moreover, the inequality is sharp in the exponents (see Remark \ref{sharp}) and yields a stability estimate in terms of the $L^1$ distance of $P$ from its convexification in the class of simple $n$-gons: denote the \textit{c-asymmetry index} of $P$ by $$\alpha_c(P):=|P \Delta P_c|,$$  where $P_c$ is a convex $n$-gon given by the Erd{\H o}s-Nagy theorem. Then, the following result holds.

\begin{corollary}\label{thm: main2}
Let $n \ge 3$ and $P$ be a simple $n$-gon. There exists $C(n)>0$ such that  
\begin{equation}\label{eq: goal2}
\alpha_c(P)+v(P) \le C(n) \delta(P).
\end{equation}
\end{corollary}

Furthermore, in the class of convex polygons, the variation completely identifies minimizers. 

\begin{corollary}\label{thm: main3}
Let $n \ge 3$ and $P$ be a convex $n$-gon. There exists $C(n)>0$ such that  
\begin{equation}\label{eq: goal2}
\sigma_s ^2(P) + \sigma_r^2(P) \le C(n) \delta(P).
\end{equation}
\end{corollary}

Theorem \ref{thm: main} yields analogous stability results for cyclic and also for equilateral polygons. This demonstrates the versatility of the lower bound given by \eqref{eq: goal}; indeed, the three quantities which comprise it (i.e. $\sigma_s$, $\sigma_r$, $\tau$) independently measure how far a given polygon deviates from the three attributes of the minimizer: cyclicity, equilaterality, and convexity.        
 
The proof is carried out in several steps. First, we consider the case when $P$ is convex and translate the problem into a functional inequality on $\RR^{2n}$ of the form $f \lesssim g$ subject to constraints involving the center of mass and the homogeneity of the variation. We show that the constraints define a compact $(2n-4)$-dimensional manifold $\mathcal{M}$ which in some sense parameterizes the class of convex polygons under investigation. The regular polygon corresponds to a point $z_* \in \mathcal{M}$ and we perform a Taylor expansion of $f$ and $g$ at $z_*$. By computing the tangent space of $\mathcal{M}$ at $z_*$ and the $2n \times 2n$ Hessian matrices of $f$ and $g$ at $z_*$, we reduce the problem to a matrix inequality. It turns out that the Hessians of $f$ and $g$ consist of blocks of circulant matrices and can be compared through delicate trigonometric matrix identities (see e.g. \eqref{eq: C-B^T B in terms of H} and \eqref{eq: D-B^T B in terms of C-B^T B and E}). The desired matrix inequality is established by performing a suitable change of coordinates and utilizing the spectral theory for circulant matrices. To finish the proof, the general case is reduced to the convex case via the Erd{\H o}s-Nagy theorem. The constant of proportionality in \eqref{eq: goal} depends on the number of sides of the polygon and the $C^3$ norms of $f$ and $g$ on $\mathcal{M}$ in a rather complicated way; nevertheless, $f$, $g$, and $\mathcal{M}$ are explicit in our construction. Last, we point out that our method of proving Theorem \ref{thm: main} may be adapted to produce other geometrically meaningful lower bounds on the  polygonal isoperimetric deficit, see Remark \ref{genth}.     

\vskip .1in 

\noindent {\bf Acknowledgements}

\vskip .05in 

\noindent We wish to thank Francesco Maggi for suggesting this line of research and Davi M\'aximo for pointing out a useful reference. Moreover, the excellent research environment provided by the University of Texas at Austin, Australian National University, MSRI, CNA, and Instituto Superior T\'ecnico is kindly acknowledged.

\section{Preliminaries}

\subsection{Setup} 
Let $n \geq 3$ and $P \subset \RR ^2$ be an $n$-gon with vertices $\{A_1, A_2,
\ldots, A_n \} \subset \RR ^2$ and center of mass $O$ which we take to be the origin. For $i \in \{1,2,\ldots,n\}$, the $i$-th side length of $P$ is $l_i:=A_iA_{i+1}$, where $A_i=A_j$ if and only if $i=j$ (mod $n$); $\{r_i:=OA_i\}_{i=1}^n$ is the set of radii; $F(P)$ is the area; $L(P):=\sum \limits_{i=1}^{n} l_i$ is the perimeter and $S(P):=\sum \limits_{i=1}^{n} l_i^2$. The variance of the sides and radii of $P$ are represented, respectively, by the quantities $$\sigma_s^2(P):=\frac{1}{n}S(P)-\frac{1}{n^2}L^2(P),$$ $$\sigma_r^2(P):=\frac{1}{n} \sum
\limits_{i=1}^{n} r_i^2 - \frac{1}{n^2} \left(\sum \limits_{i=1}^{n}
r_i\right)^2.$$

\subsection{Circulant matrices}

The key objects in our analysis are the so-called circulant matrices which arise in various branches of mathematics, see e.g. \cite{gray'06}. In what follows, we recall some basic properties which will be employed in our study. Let $\{a_0,a_1,\ldots,a_{n-1}\} \subset \CC$ be a given set of complex numbers; then, the matrix 
\begin{equation}\label{eq: circ_matr_def}
A=
  \begin{pmatrix}
  a_0 & a_1& a_2 & \cdots & a_{n-1}\\
  a_{n-1} & a_0&a_1&\cdots &a_{n-2}\\
  \vdots & \ddots& \ddots& \ddots& \vdots\\
  a_2 &a_3 & \cdots& a_0&a_1\\
  a_1 & a_2 &\cdots &a_{n-1}& a_0
\end{pmatrix}
\end{equation}
is called the circulant matrix generated by $\{a_0,a_1,\ldots,a_{n-1}\}$. Note that to form this matrix, one simply takes $(a_0,a_1,\ldots,a_{n-1})$ as the first row and thereafter cyclically permutes the entries to obtain the next row. Circulant matrices enjoy many useful properties, one of which is that their eigenvalues and eigenvectors are explicit. Let $\{\omega _k:=e^{\frac{2 \pi i k}{n}}\}_{k=0}^{n-1}$ be the $n$-th roots of unity and set
\begin{equation}\label{eq: eigs_of_circ_matr}
  \psi _k:=\sum \limits _{j=0}^{n-1} a_j \omega _k^j.
\end{equation}
It is not difficult to check that the eigenvalues of the matrix defined in \eqref{eq: circ_matr_def} are the complex numbers $\{ \psi _k \}_{k=0}^{n-1}$ given by \eqref{eq: eigs_of_circ_matr}. Furthermore, the eigenvector corresponding to $\psi _k$ is given by $u_k:=(1,\omega _k,\omega _k^2,\ldots,\omega _k^{n-1})$, and the set $\{u_0,u_1,\ldots,u_{n-1}\}$ forms a complex orthogonal basis in $\CC ^n$.
Note that all circulant matrices share the  same orthogonal basis of eigenvectors. In our analysis, we shall make use of the following result concerning real symmetric circulant matrices. The proof is elementary and we present it solely for the reader's convenience. 


\begin{proposition}\label{prp: eigs_real_sym_circ_matr}
Let $A$ be the matrix given by \eqref{eq: circ_matr_def}. If $A$ is real and symmetric, then the eigenvalues of $A$ satisfy $\psi _k=\psi _{n-k}$ for all $k \in \{1,\ldots,n-1\}$. Moreover, let $v_0:=(1,\ldots,1)$ and for $l \in \{1,\ldots, \lfloor \frac{n}{2} \rfloor \}$, define
   \begin{align}
     v_{2l-1}&:=\left(1,\cos \frac{2 \pi l}{n},\cos \frac{4 \pi l}{n},\ldots,\cos \frac{2 \pi l (n-1)}{n}\right),\nonumber\\ 
     v_{2l}&:=\left(0,\sin \frac{2\pi l}{n},\sin \frac{4\pi l}{n},\ldots,\sin \frac{2 \pi l(n-1)}{n}\right).
   \end{align} 
Then, $v_k$ is an eigenvector of $A$ corresponding to the eigenvalue $\psi _{\lceil \frac{k}{2} \rceil}$, and the set $\{v_0,v_1,\ldots,v_{n-1}\}$ forms a real orthogonal basis in $\RR ^n$.
 \end{proposition}
 \begin{proof}
First, since $A$ is symmetric, all its eigenvalues are real. Moreover, $\omega _{n-k}=\overline{\omega _k}$ and since $\{a_j\} \subset \RR$, we have 
   \begin{equation}
          \psi _{n-k}=\sum \limits _{j=0}^{n-1} a_j \omega _{n-k}^j=\sum \limits _{j=0}^{n-1} a_j \overline{\omega _k}^j=\overline{\psi _k}=\psi_k.
   \end{equation}
Now let $u_k=(1,\omega _k,\omega _k^2,\ldots,\omega _k^{n-1})$, and note that 
   \begin{equation*}
     v_{2l-1}=\frac{u_l+u_{n-l}}{2}, \quad v_{2l}=\frac{u_l-u_{n-l}}{2i},
   \end{equation*}
for $l \in \{1,\ldots, \lfloor \frac{n}{2} \rfloor \}$. Therefore, the vectors $v_{2l-1}$ and $v_{2l}$ are eigenvectors corresponding to the eigenvalue $\psi _l$; note also that $v_0=u_0$. It remains to prove that the $v_k$'s are mutually orthogonal. Since $A$ is a circulant matrix, $u_k \perp u_{k'}$, whenever $k \neq k'$. Hence, $v_{2l'}, v_{2l'-1} \perp v_{2l},v_{2l-1}$, for $1\leq l \neq l' \leq \frac{n}{2}$. Moreover, $v_k \perp v_0$, for all $1\leq k \leq n-1$, and for $1\leq l \leq \frac{n}{2}$,     
\begin{align*}
     \langle v_{2l-1},v_{2l} \rangle &=\langle \frac{u_l+u_{n-l}}{2},\frac{u_l-u_{n-l}}{2i} \rangle \\
     &=\frac{1}{4i}\left( |u_l|^2-|u_{n-l}|^2-2i \Im \langle u_l,u_{n-l}\rangle \right)=0;
   \end{align*}
 thus, $v_{2l} \perp v_{2l-1}$.   
 \end{proof}
 
 \begin{remark}
All real symmetric circulant matrices share the same real orthogonal basis of eigenvectors. 
 \end{remark}

\section{Proof of Theorem \ref{thm: main}}

The proof will be split up into two cases. First, we handle the convex case in \S \ref{sec1} - \S \ref{lastsec}, and then consider the general case in \S \ref{gen}.

\subsection{A functional formulation} \label{sec1}
In what follows, the dependence on $P$ will often be omitted to simplify the notation. The starting point is the following well-known inequality (see \cite[pg. 35]{shilleto'85}) which holds for any $n$-gon:
\begin{equation}\label{eq: shiletto1}
 8n^2  \sigma_r ^2 \sin ^2 \frac{\pi}{n} \le nS-c_n F;
\end{equation}
since $n^2 \sigma_s ^2=nS-L^2$, it follows that \eqref{eq: shiletto1} is equivalent to 
\begin{equation} \label{fin}
8n^2  \sin ^2 \frac{\pi}{n} \ \sigma_r ^2 \le   \delta+n^2 \sigma_s ^2 .
\end{equation}
Therefore, in order to establish \eqref{eq: goal} for convex $n$-gons, it suffices to prove that $\sigma_s ^2(P) \le C(n) \delta(P) $ for some positive constant $C(n)$. However, this is equivalent to showing that the ratio $$\frac{nS-L^2}{L^2-c_nF}$$ is bounded; in particular, it suffices to show that the ratio
\begin{equation} \label{ratio}
\frac{nS-c_n F}{L^2-c_n F}=\frac{nS-L^2}{L^2-c_nF}+1
\end{equation}
is bounded. Let $x_i$ be the angle between radii $OA_i$ and $OA_{i+1}$, for
$i=1,2,\ldots,n$. Since $P$ is convex, we have that $\sum \limits _{i=1}^n x_i=2 \pi$. Furthermore, 
\begin{equation}\label{eq: SLF_formulas}
  \begin{cases}
S=\sum  \limits_{i=1}^{n} l_i^2=\sum \limits_{i=1}^{n} \left( r_{i+1}^2+r_i^2-2r_{i+1}r_i \cos
x_i\right);\\
L=\sum  \limits_{i=1}^{n} l_i=\sum \limits_{i=1}^{n} \left( r_{i+1}^2+r_i^2-2r_{i+1}r_i \cos
x_i\right)^{1/2};\\
F=\frac{1}{2} \sum \limits_{i=1}^n r_i r_{i+1} \sin x_i.
\end{cases}
\end{equation}
Using these formulas we obtain
\begin{align}\label{eq: NS-c_nF}
nS-4n \tan \frac{\pi}{n}F=n\sum \limits_{i=1}^{n} \left( r_{i+1}^2+r_i^2-2r_{i+1}r_i \cos
x_i\right)-2n \tan \frac{\pi}{n} \sum \limits_{i=1}^n r_i r_{i+1} \sin x_i,
\end{align}
\begin{align}\label{eq: L2-c_nF}
L^2-4n \tan \frac{\pi}{n}F=\left(\sum \limits_{i=1}^{n} \left( r_{i+1}^2+r_i^2-2r_{i+1}r_i \cos
x_i\right)^{1/2}\right)^2-2n  \tan \frac{\pi}{n} \sum \limits_{i=1}^n r_i r_{i+1} \sin x_i.
\end{align}
Let
\begin{align}\label{eq: f}
f(x_1,x_2,\ldots,x_n;r_1,r_2,\ldots,r_n)&:=n\sum \limits_{i=1}^{n} \left( r_{i+1}^2+r_i^2-2r_{i+1}r_i \cos
x_i\right)\\ \nonumber
&-2n \tan \frac{\pi}{n} \sum \limits_{i=1}^n r_i r_{i+1} \sin x_i,
\end{align}
and
\begin{align}\label{eq: g}
g(x_1,x_2,\ldots,x_n;r_1,r_2,\ldots,r_n)&:= \left(\sum \limits_{i=1}^{n} \left( r_{i+1}^2+r_i^2-2r_{i+1}r_i \cos
x_i\right)^{1/2}\right)^2\\ \nonumber
&-2n \tan \frac{\pi}{n} \sum \limits_{i=1}^n r_i r_{i+1} \sin x_i.
\end{align}
By setting $(x;r):=(x_1,x_2,\ldots,x_n;r_1,r_2,\ldots,r_n)$, we note that in order to prove \eqref{eq: goal}, it suffices to prove that

\begin{equation}\label{eq: f<Cg}
  f(x;r) \leq C(n) \ g(x;r)
\end{equation}
for all $(x;r) \in \RR ^{2n}$ satisfying
\begin{equation}\label{eq: sum_x_i=2pi}
  \sum \limits _{i=1}^{n} x_i=2 \pi.
\end{equation}

\noindent Moreover, since $f$ and $g$ are 2-homogeneous in the $r$ variable, we may assume without loss of generality that
\begin{equation}\label{eq: sum_r^2=n}
  \sum \limits _{i=1}^{n} r_i=n.
\end{equation}

\noindent Next, note that a point $O$ is the centroid of $P$ if and only if
\begin{equation*}
  \sum \limits _{i=1}^{n} \overrightarrow{OA_i}=0,
\end{equation*}
which is equivalent to saying that the projections of $\sum \limits _{i=1}^{n} \overrightarrow{OA_i}$ onto $\overrightarrow{OA_1}$ and $\overrightarrow{OA_1}^\perp$ vanish. The projection in the $\overrightarrow{OA_1}$ direction is
\begin{equation*}
  \sum \limits _{i=1}^{n} r_i \cos \left(\sum \limits_{k=1}^{i-1} x_k \right),
\end{equation*}
where an empty sum is to be understood as $0$. Similarly the projection in the orthogonal direction $\overrightarrow{OA_1}^\perp$ is
\begin{equation*}
  \sum \limits _{i=1}^{n} r_i \sin \left(\sum \limits_{k=1}^{i-1} x_k \right).
\end{equation*}
Hence, 
\begin{equation}\label{eq: centroid_cond}
  \begin{cases}
  \sum \limits _{i=1}^{n} r_i \cos \left(\sum \limits_{k=1}^{i-1} x_k \right)=0,\\
  \sum \limits _{i=1}^{n} r_i \sin \left(\sum \limits_{k=1}^{i-1} x_k \right)=0.
  \end{cases}
\end{equation}

\noindent Note that subject to the constraints \eqref{eq: sum_x_i=2pi}, \eqref{eq: sum_r^2=n}, \eqref{eq: centroid_cond},  the regular $n$-gon corresponds to the point $(x_*; r_*)=\left(\frac{2\pi}{n},\ldots,\frac{2\pi}{n};1,\ldots,1\right)$.
Let
\begin{equation}
  \mathcal{M}:=\Big \{ (x;r) \in \RR ^{2n}:\ x_i,r_i \geq 0,\ \eqref{eq: sum_x_i=2pi},\ \eqref{eq: sum_r^2=n},\ \eqref{eq: centroid_cond} \ \mbox{hold} \Big\},
\end{equation}
and observe that $\mathcal{M}$ is a compact manifold of dimension $2n-4$ and all convex $n$-gons with centroid $(0,0) \in \RR ^2$ have a representation as points $(x;r)\in \mathcal{M}$ where $(x;r)$ is associated with the $n$-gon whose $i$-th vertex has distance $r_i$ from the origin and two consecutive vertices form an angle $x_i$. Thus, in order to prove that the ratio $\frac{nS-c_n F}{L^2-c_n F}$ is bounded, it suffices to establish inequality \eqref{eq: f<Cg} on the set $\mathcal{M}$. The next step consists of localizing the problem.

\subsection{Localization}

The polygonal isoperimetric inequality implies $g \geq 0$ with equality if and only if $P$ is the regular $n$-gon; therefore, $g(x;r)=0$ for $(x,r) \in \mathcal{M}$ if and only if $(x;r)=(x_*;r_*)$. By \cite[(4.1)]{shilleto'85}, the same is true for the function $f$. Therefore, since $f$ and $g$ are continuous, for every neighborhood $B_{\delta}$ of the point $(x_*;r_*)$, it follows that if $(x;r) \in \mathcal{M} \setminus B_{\delta}$, then 
\begin{equation*}
  f(x;r) \leq C \ g(x;r),
\end{equation*}
where $$0<C:= \frac{\sup \limits_{\mathcal{M} \setminus B_{\delta}} f}{\inf \limits_{\mathcal{M} \setminus B_{\delta}} g}<\infty.$$
Therefore, in order to prove inequality \eqref{eq: f<Cg}, it suffices to prove it for some neighborhood $B_{\delta}$ of the point $(x_*;r_*)$.

\subsection{Tangent space of $\mathcal{M}$ at $(x_*;r_*)$}

Let $\HH$ be the tangent space of the manifold $\mathcal{M}$ at the point $(x_*;r_*)$. To simplify the notation, set $z:=(x;r)$ and $z_*=(x_*;r_*)$. Furthermore, let
\begin{align}
  h_1(x;r)&:=\sum \limits_{i=1}^{n} x_i,\\ \nonumber
  h_2 (x;r)&:=\sum \limits_{i=1}^{n} r_i,\\ \nonumber
  h_3(x;r)&:=\sum \limits _{i=1}^{n} r_i \cos \left(\sum \limits_{k=1}^{i-1} x_i \right), \\ \nonumber
  h_4(x;r)&:=\sum \limits _{i=1}^{n} r_i \sin \left(\sum \limits_{k=1}^{i-1} x_i \right).
\end{align}

\noindent Then, $\mathcal{M}$ is defined by the equations $h_i(x;r)=0$ for $i=1,2,3,4$,
therefore $\HH$ is given by
\begin{equation*}
\langle (x;r), \nabla h_i(z_*)\rangle=0, \ i=1,2,3,4;
\end{equation*}
thus, to identify $\HH$, we compute the gradients of the functions $h_i$ at the point $z_*$: 

\begin{align}\label{eq: grads of h1,h2}
  \nabla h_1(z_*)=(1,1,\ldots,1;0,0,\ldots,0),\\ \nonumber
  \nabla h_2(z_*)=(0,0,\ldots,0;1,1,\ldots,1),
\end{align}
\begin{equation}\label{eq: grad of h3}
  \begin{cases}
    D_{r_k}h_3(z_*)=\cos \frac{2\pi(k-1)}{n},\\
    D_{x_k}h_3(z_*)=\frac{\cos \frac{\pi}{n}-\cos \frac{\pi(2k-1)}{n}}{2 \sin \frac{\pi}{n}}, \ k=1,2,\ldots,n,
  \end{cases}
\end{equation}
\begin{equation}\label{eq: grad of h4}
  \begin{cases}
    D_{r_k}h_4(z_*)=\sin \frac{2\pi(k-1)}{n},\\
    D_{x_k}h_4(z_*)=-\frac{\sin \frac{\pi}{n}+\sin \frac{\pi(2k-1)}{n}}{2 \sin \frac{\pi}{n}}, \ k=1,2,\ldots,n.
  \end{cases}
\end{equation}
Therefore, $\HH$ is given by
\begin{align*}
  \sum \limits_{k=1}^{n} x_k&=0,\\
  \sum \limits_{k=1}^{n} r_k&=0,\\
  \sum \limits_{k=1}^{n} \frac{\cos \frac{\pi}{n}-\cos \frac{\pi(2k-1)}{n}}{2 \sin \frac{\pi}{n}} x_k+\sum \limits_{k=1}^{n} r_k \cos \frac{2\pi(k-1)}{n}&=0,\\
  -\sum \limits_{k=1}^{n} \frac{\sin \frac{\pi}{n}+\sin\frac{\pi(2k-1)}{n}}{2 \sin \frac{\pi}{n}} x_k+\sum \limits_{k=1}^{n} r_k \sin \frac{2\pi(k-1)}{n}&=0,
\end{align*}
or equivalently by
\begin{align}\label{eq: tangent space H}
  \sum \limits_{k=1}^{n} x_k&=0,\\ \nonumber
  \sum \limits_{k=1}^{n} r_k&=0,\\ \nonumber
  -\sum \limits_{k=1}^{n} \frac{\cos \frac{\pi(2k-1)}{n}}{2 \sin \frac{\pi}{n}} x_k+\sum \limits_{k=1}^{n} r_k \cos \frac{2\pi(k-1)}{n}&=0,\\ \nonumber
  -\sum \limits_{k=1}^{n} \frac{\sin\frac{\pi(2k-1)}{n}}{2 \sin \frac{\pi}{n}} x_k+\sum \limits_{k=1}^{n} r_k \sin \frac{2\pi(k-1)}{n}&=0.
\end{align}

\subsection{Taylor expansion}

By expanding the functions $f,g$ into Taylor series around the point $z_*=(x_*;r_*)$, we have that
\begin{align}\label{eq: f_taylor}
  f(z)=&f(z_*)+Df(z_*)(z-z_*)+\frac{1}{2}\langle D^2f(z_*)(z-z_*) , (z-z_*) \rangle\\ \nonumber
  &+\frac{1}{6} \sum \limits _{i,j,k=1}^{2n} D_{ijk} f((1-\theta _z) z_*+\theta _z z)(z-z_*)_i(z-z_*)_j(z-z_*)_k,
 \end{align}
and
\begin{align}\label{eq: g_taylor}
  g(z)=&g(z_*)+Dg(z_*)(z-z_*)+\frac{1}{2}\langle D^2g(z_*)(z-z_*) , (z-z_*) \rangle\\ \nonumber
  &+\frac{1}{6} \sum \limits _{i,j,k=1}^{2n} D_{ijk} g((1-\tau _z) z_*+\tau _z z)(z-z_*)_i(z-z_*)_j(z-z_*)_k,
 \end{align}
 for some $\theta _z, \tau _z \in (0,1)$. Next, we establish a strategy of obtaining  \eqref{eq: f<Cg} by computing second derivatives of the functions $f$ and $g$ and reducing the problem to certain matrix inequalities.

\begin{lemma} \label{cond}
Suppose the following conditions hold:
\begin{itemize}
  \item[(i)] $f(z_*)=g(z_*)=0$;
  \item[(ii)] $Df(z_*)(z-z_*)=Dg(z_*)(z-z_*)=0$, for $z$ satisfying \eqref{eq: sum_x_i=2pi} and \eqref{eq: sum_r^2=n};
  \item[(iii)] $\langle D^2 f(z_*)(z-z_*) , (z-z_*) \rangle \leq C\langle D^2g(z_*)(z-z_*) , (z-z_*) \rangle$, for $z$ satisfying \eqref{eq: sum_x_i=2pi} and \eqref{eq: sum_r^2=n};
  \item[(iv)] $\langle D^2f(z_*)w , w \rangle >0$, for $w\neq 0$ such that $w \in \HH$.
\end{itemize}
Then, \eqref{eq: f<Cg} holds in some neighborhood of $z_*$.
\end{lemma}

\begin{proof}
Since $f$ and $g$ are $C^3$ and $\mathcal{M}$ is compact, it follows that
\begin{align}\label{eq: D3f}
  \frac{1}{6} \sum \limits _{i,j,k=1}^{2n} D_{ijk} f((1-\theta _z) z_*+\theta _z z)(z-z_*)_i(z-z_*)_j(z-z_*)_k \leq C_\mathcal{M} |z-z_*|^3,
\end{align}
and
\begin{align}\label{eq: D3g}
  \frac{1}{6} \sum \limits _{i,j,k=1}^{2n} D_{ijk} g((1-\theta _z) z_*+\theta _z z)(z-z_*)_i(z-z_*)_j(z-z_*)_k \leq C_\mathcal{M} |z-z_*|^3,
\end{align}
for $z \in \mathcal{M}$, where $C_\mathcal{M}>0$. By compactness and (iv), we have
\begin{equation*}
  \inf_{w \in S_{\HH}} \langle D^2f(z_*)w , w \rangle =: \sigma>0,
\end{equation*}
where $S_{\HH}$ is the unit sphere in the subspace $\HH$ (with center $z_*$). Moreover, by continuity, there exists a neighborhood $U \subset \RR ^{2n}$ of $S_{\HH}$ such that
\begin{equation} \label{nondeg}
  \langle D^2f(z_*)w , w \rangle \geq \frac{\sigma}{2},
\end{equation}
for all $w \in U$.
Next, note that for $z \in \mathcal{M}$ sufficiently close to $z_*$, we have $\frac{z-z_*}{|z-z_*|} \in U$. Hence, \eqref{nondeg}, (i), (ii), (iii), \eqref{eq: D3f}, and \eqref{eq: D3g} imply that there exists a neighborhood $V \subset \mathcal{M}$ of $z_*$ such that if $z \in V$, then
\begin{equation*}
  f(z) \leq \langle D^2f(z_*)(z-z_*) , (z-z_*)\rangle,
\end{equation*}
and
\begin{equation*}
  g(z) \geq \frac{1}{4}\langle D^2g(z_*)(z-z_*) , (z-z_*)\rangle.
\end{equation*}
To conclude, note that (iii) implies 
\begin{equation*}
  f(z) \leq 4C g(z),
\end{equation*}
for $z \in V$.
\end{proof}

\noindent To complete the proof of Theorem \ref{thm: main} for convex $n$-gons, we verify conditions (i)-(iv) in Lemma \ref{cond}. Note that (i) holds as a result of the polygonal isoperimetric inequality. The remaining sections are devoted to checking (ii)-(iv).    

\subsection{Derivatives of $f$ and $g$}

In this section, we compute the derivatives of $f$ and $g$ at the point $z_*=(x_*;r_*)=(\frac{2\pi}{n},\frac{2\pi}{n},\ldots,\frac{2\pi}{n};1,1,\ldots,1)$. Note that $$\frac{\cos(\pi /n)}{2n} f(x;r)=\cos \frac{\pi}{n} \ \sum \limits _{i=1}^n  r_i^2 - \sum \limits _{i=1}^n r_i r_{i+1} \cos(x_i-\frac{\pi}{n}).$$ By a slight abuse of notation, we denote the right-hand side by f, i.e. from now on, $$f(x;r)= \cos \frac{\pi}{n} \ \sum \limits _{i=1}^n r_i^2 - \sum \limits _{i=1}^n r_i r_{i+1} \cos(x_i-\frac{\pi}{n}).$$ Note that our notation is periodic modulo $n$, therefore differentiation with respect to the $r_{n+k},x_{n+k}$ variables is the same as differentiation with respect to $r_k$ and $x_k$, respectively. By direct computation, 
\begin{equation}\label{eq: Df}
\begin{cases}
D_{x_i}f=\sin \frac{\pi}{n},\\
D_{r_i}f=0.
\end{cases}
\end{equation}

\begin{equation}\label{eq: D2_xx f}
\begin{cases}
D_{x_i x_i}f=\cos \frac{\pi}{n},\\
D_{x_i x_j}f=0,\ \mbox{for} \ i \neq j.
\end{cases}
\end{equation}

\begin{equation}\label{eq: D2_rr f}
\begin{cases}
D_{r_i r_i}f=2 \cos \frac{\pi}{n},\\
D_{r_i r_j}f=-\cos \frac{\pi}{n},\ \mbox{for} \ |i - j|=1.
\end{cases}
\end{equation}

\begin{equation}\label{eq: D2_xr f}
\begin{cases}
D_{x_i r_j}f= \sin \frac{\pi}{n},\ \mbox{for} \ j=i,i+1\\
D_{x_i r_j}f=0,\ \mbox{otherwise}.
\end{cases}
\end{equation}

Furthermore, 
\begin{equation}\label{eq: Dg}
\begin{cases}
D_{x_i}g= 2n \tan \frac{\pi}{n},\\
D_{r_i}g=0.
\end{cases}
\end{equation}

\begin{equation}\label{eq: D2_xx g}
\begin{cases}
D_{x_i x_i}g= 2 \cos ^2 \frac{\pi}{n}+2n \sin ^2 \frac{\pi}{n},\\
D_{x_i x_j}g=2 \cos ^2 \frac{\pi}{n},\ \mbox{for} \ j \neq i.
\end{cases}
\end{equation}

\begin{equation}\label{eq: D2_rr g}
\begin{cases}
D_{r_i r_i}g= 8 \sin ^2 \frac{\pi}{n}+4n \cos ^2 \frac{\pi}{n},\\
D_{r_i r_j}g= (8-4n) \sin ^2 \frac{\pi}{n}-2n \cos ^2 \frac{\pi}{n},\ \mbox{for} \  |j - i|=1,\\
D_{r_i r_j}g= 8 \sin ^2 \frac{\pi}{n}, \ \mbox{otherwise}.
\end{cases}
\end{equation}

\begin{equation}\label{eq: D^2_xr g}
\begin{cases}
D_{x_i r_j}g= 2 \sin \frac{2\pi}{n}+2n \sin ^2 \frac{\pi}{n}\tan \frac{\pi}{n},\ \mbox{for} \  j=i,i+1,\\
D_{x_i r_j}g= 2 \sin \frac{2\pi}{n}, \ \mbox{otherwise}.
\end{cases}
\end{equation}

\subsection{Condition (ii)}

By \eqref{eq: Df}, 
\begin{align*}
Df(x_*;r_*)(x-x_*;r-r_*)&=D_x f(x_*;r_*)(x-x_*)+D_r f(x_*;r_*)(r-r_*)\\
&=\sin \frac{\pi}{n} \sum \limits _{i=1}^{n}(x_i-(x_*)_i)=0,
\end{align*}
since $\sum \limits _{i=1}^{n} x_i=\sum \limits _{i=1}^{n} (x_*)_i= 2 \pi$. Similarly from \eqref{eq: Dg},
\begin{align*}
Dg(x_*;r_*)(x-x_*;r-r_*)&=D_x g(x_*;r_*)(x-x_*)+D_r g(x_*;r_*)(r-r_*)\\
&=2n \tan \frac{\pi}{n} \sum \limits _{i=1}^{n}(x_i-(x_*)_i)=0.
\end{align*}

\subsection{Conditions (iii) and (iv)} \label{lastsec}

Let $F:=D^2f(x_*;r_*)$, $G:=D^2g(x_*;r_*)$, and note that $F,G \in M_{2n \times 2n}(\RR)$ are square symmetric matrices. Consider the subspace of $\HH$ given by
\begin{equation}
  \HH _1:=\Bigg\{(x;r) :\ \sum \limits_{i=1}^n x_i=0,\ \sum \limits_{i=1}^n r_i=0 \Bigg \},
\end{equation}
and note that condition (iii) is precisely the matrix inequality
\begin{equation}\label{eq: F<G}
F \leq C G
\end{equation}
in $\HH _1$. To prove \eqref{eq: F<G}, we utilize a suitable coordinate transformation such that in the new coordinates system, the quadratic forms associated to the matrices $F$ and $G$ take a substantially simpler form. Note that
\begin{equation}\label{eq: F}
  F=\cos \frac{\pi}{n}\begin{pmatrix}
  I & B \\
  B^T & K
 \end{pmatrix},
\end{equation}
where $I$ is the $n \times n$ identity matrix,
\begin{equation}\label{eq: B}
B=\begin{pmatrix}
  \tan \frac{\pi}{n} &  \tan \frac{\pi}{n} & 0 & \cdots & 0\\
  0 & \tan \frac{\pi}{n} &  \tan \frac{\pi}{n} & \cdots & 0\\
  \vdots & 0& \ddots& \ddots& \vdots\\
  0 &\cdots &0 & \tan \frac{\pi}{n}& \tan \frac{\pi}{n}\\
  \tan \frac{\pi}{n} & 0 & \cdots & 0 & \tan \frac{\pi}{n}\\
 \end{pmatrix} _{n \times n},
 \end{equation}
and
\begin{equation}\label{eq: C}
K=\begin{pmatrix}
  2 &  -1 & 0 & \cdots &0 & -1\\
  -1 & 2 &  -1 & 0&\cdots & 0\\
  0& -1 & 2 & -1 & \ddots & \vdots\\
  \vdots &0 & \ddots & \ddots & \ddots & 0\\
  0 &\vdots& \ddots & -1& 2& -1\\
  -1 & 0 & \cdots & 0 & -1 & 2\\
 \end{pmatrix} _{n \times n}.
 \end{equation}
Let $M:=\begin{pmatrix}
  I & B \\
  B^T & K
 \end{pmatrix}$ and note that since $M$ is a constant multiple of the matrix $F$, inequality \eqref{eq: F<G} is equivalent to 
 \begin{equation}\label{eq: M<G}
   M \leq C G,
 \end{equation}
in the subspace $\HH _1$ for some constant $C >0$. Next, we consider the matrix $G$ and construct a matrix $G'$, which is of a simpler form than $G$ but is \emph{equivalent} to $G$ in the sense that the quadratic forms associated to $G$ and $G'$ are equal in the subspace $\HH _1$. By\eqref{eq: D2_xx g}, \eqref{eq: D2_rr g}, and \eqref{eq: D^2_xr g} we have that the quadratic form associated to $G$ has the form
 \begin{align}
   \langle G(x;r),(x;r) \rangle & =\left(2 \cos^2 \frac{\pi}{n}+2n\sin ^2 \frac{\pi}{n}\right) \sum \limits_{i=1}^n x_i^2+2 \cos^2 \frac{\pi}{n} \sum \limits_{i \neq j} x_i x_j +\\ \nonumber
   & \left(8 \sin ^2 \frac{\pi}{n}+4n \cos^2 \frac{\pi}{n}\right) \sum \limits_{i=1}^n r_i^2 + 8 \sin ^2 \frac{\pi}{n} \sum \limits_{|i-j|>1} r_i r_j + \\ \nonumber
   &\left((8-4n) \sin ^2 \frac{\pi}{n}-2n \cos^2 \frac{\pi}{n}\right) \sum \limits_{|i-j|=1} r_i r_j +\\ \nonumber
   & \left(4\sin\frac{2\pi}{n}+4n \tan \frac{\pi}{n} \sin ^2 \frac{\pi}{n}\right)\sum \limits_{j-i=0,1} x_i r_j+4\sin\frac{2\pi}{n} \sum \limits_{j-i\neq 0,1} x_i r_j\\ \nonumber
   &=2 \cos^2 \frac{\pi}{n} \left(\sum \limits_{i=1}^n x_i\right)^2+2n\sin ^2 \frac{\pi}{n} \sum \limits_{i=1}^n x_i^2+8 \sin ^2 \frac{\pi}{n} \left(\sum \limits_{i=1}^n r_i\right)^2 +\\ \nonumber
   & 4n \cos^2 \frac{\pi}{n} \sum \limits_{i=1}^n r_i^2-\left(4n \sin ^2 \frac{\pi}{n}+2n \cos^2 \frac{\pi}{n}\right) \sum \limits_{|i-j|=1} r_i r_j+ \\ \nonumber
   & 4n \tan \frac{\pi}{n} \sin ^2 \frac{\pi}{n} \sum \limits_{j-i=0,1} x_i r_j + 4\sin \frac{2\pi}{n} \sum \limits _{i=1}^{n} x_i \cdot \sum \limits _{j=1}^{n} r_j.
 \end{align}
Since $(x;r) \in \HH _1$, it follows that
\begin{align}
  \langle G(x;r),(x;r) \rangle &=2n\sin ^2 \frac{\pi}{n} \sum \limits_{i=1}^n x_i^2+4n \tan \frac{\pi}{n} \sin ^2 \frac{\pi}{n} \sum \limits_{j-i=0,1} x_i r_j + \\ \nonumber
  & 4n \cos^2 \frac{\pi}{n} \sum \limits_{i=1}^n r_i^2-\left(4n \sin ^2 \frac{\pi}{n}+2n \cos^2 \frac{\pi}{n}\right) \sum \limits_{|i-j|=1} r_i r_j \\ \nonumber
  &= \langle G'(x;r),(x;r) \rangle ,\\ \nonumber
\end{align}
where
\begin{equation}\label{eq: G'}
  G'=2n \sin ^2 \frac{\pi}{n}\begin{pmatrix}
  I & B \\
  B^T & D
 \end{pmatrix}.
\end{equation}
Here $I$ and $B$ are as before and $D$ is given by
\begin{equation}\label{eq: C}
D=\begin{pmatrix}
  \frac{2}{\tan ^2 \frac{\pi}{n}} &  -2-\frac{1}{\tan^2 \frac{\pi}{n}} &  & \cdots & & -2-\frac{1}{\tan^2 \frac{\pi}{n}}\\
  -2-\frac{1}{\tan^2 \frac{\pi}{n}} & \frac{2}{\tan ^2 \frac{\pi}{n}} &  -2-\frac{1}{\tan^2 \frac{\pi}{n}} & &\cdots & \\
  & -2-\frac{1}{\tan^2 \frac{\pi}{n}} & \frac{2}{\tan ^2 \frac{\pi}{n}} & -2-\frac{1}{\tan^2 \frac{\pi}{n}} & \ddots & \vdots\\
  \vdots & & \ddots & \ddots & \ddots & \\
   &\vdots& \ddots & -2-\frac{1}{\tan^2 \frac{\pi}{n}}& \frac{2}{\tan ^2 \frac{\pi}{n}}& -2-\frac{1}{\tan^2 \frac{\pi}{n}}\\
  -2-\frac{1}{\tan^2 \frac{\pi}{n}} &  & \cdots &  & -2-\frac{1}{\tan^2 \frac{\pi}{n}} & \frac{2}{\tan ^2 \frac{\pi}{n}}\\
 \end{pmatrix} _{n \times n}.
 \end{equation}
Let
\begin{equation}\label{eq: N}
  N=\begin{pmatrix}
  I & B \\
  B^T & D
 \end{pmatrix}
\end{equation}
and note that since $G = G'=2n \sin ^2 \frac{\pi}{n} N$ in $\HH_1$, if
 \begin{equation}\label{eq: M<N}
   M \leq C N,
 \end{equation}
 in $\HH _1$, then the desired inequality \eqref{eq: F<G} follows. To this aim, let
\begin{equation}\label{eq: Q}
  Q=\begin{pmatrix}
  I & -B \\
  0 & I
 \end{pmatrix},
\end{equation}
and note that $Q$ is non-degenerate: $\det Q =1$, and
\begin{equation}
  Q^T M Q=\begin{pmatrix}
  I & 0 \\
  0 & K-B^T B
 \end{pmatrix},
\end{equation}
\begin{equation}
  Q^T N Q=\begin{pmatrix}
  I & 0 \\
  0 & D-B^T B
 \end{pmatrix}.
 \end{equation}
 Therefore, consider new coordinates $(\zeta;\eta)$ such that
 \begin{equation}\label{eq: coordtransform}
 (x;r)=Q(\zeta;\eta).
 \end{equation}
 In this coordinate system, we have that
 \begin{equation*}
   \langle M(x;r),(x;r) \rangle = \langle Q^T M Q(\zeta;\eta),(\zeta;\eta) \rangle,
 \end{equation*}
 and
 \begin{equation*}
   \langle N(x;r),(x;r) \rangle = \langle Q^T N Q(\zeta;\eta),(\zeta;\eta) \rangle.
 \end{equation*}
 Hence, the inequality $M \leq C N$ is equivalent to $Q^T M Q \leq C Q^T N Q$. Since $x=\zeta-B \eta$ and $r=\eta$, it follows that under the above coordinate transformation, $\HH _1$ maps to the subspace given by $$\sum \limits _{i=1}^{n} \zeta _i = 2 \tan \frac{\pi}{n}\sum \limits _{i=1}^{n} \eta _i,$$ $$\sum \limits _{i=1}^{n} \eta _i=0,$$ or equivalently, $\sum \limits _{i=1}^{n} \zeta _i =  \sum \limits _{i=1}^{n} \eta _i=0$, which is in fact $\HH _1$ itself. This means that $\HH _1$ is invariant under the coordinate transformation given by \eqref{eq: coordtransform}. Therefore, it remains to prove 
 \begin{equation}\label{eq: QM<QN}
   Q^T M Q \leq C Q^T N Q
 \end{equation}
in $\HH _1$. 

\noindent Next, we turn our attention to $K-B^T B$ and $D-B^T B$. Plugging in the formulas for the matrices $K,B,D$, it follows that
 \begin{equation}\label{eq: C_B^T B}
K-B^T B=
\left(\begin{smallmatrix}
  2-2\tan^2 \frac{\pi}{n} &  -1-\tan ^2 \frac{\pi}{n} & 0 & \cdots &0 & -1-\tan ^2 \frac{\pi}{n}\\
  -1-\tan ^2 \frac{\pi}{n} & 2-2\tan^2 \frac{\pi}{n} &  -1-\tan ^2 \frac{\pi}{n} & 0&\cdots & 0\\
  0& -1-\tan ^2 \frac{\pi}{n} & 2-2\tan^2 \frac{\pi}{n} & -1-\tan ^2 \frac{\pi}{n} & \ddots & \vdots\\
  \vdots &0 & \ddots & \ddots & \ddots & 0\\
  0 &\vdots& \ddots & -1-\tan ^2 \frac{\pi}{n}& 2-2\tan^2 \frac{\pi}{n}& -1-\tan ^2 \frac{\pi}{n}\\
  -1-\tan ^2 \frac{\pi}{n} & 0 & \cdots & 0 & -1-\tan ^2 \frac{\pi}{n} & 2-2\tan^2 \frac{\pi}{n}\\
 \end{smallmatrix} \right)_{n \times n},
 \end{equation}
 and\\
 \\
 \begin{equation}\label{eq: D_B^T B}
D-B^T B=
\end{equation}
$$
\scalefont{0.4}\left(\begin{smallmatrix}
  \frac{2}{\tan ^2 \frac{\pi}{n}}-2\tan^2 \frac{\pi}{n} &  -2-\frac{1}{\tan^2 \frac{\pi}{n}}-\tan ^2 \frac{\pi}{n} & & \cdots & & -1-\tan ^2 \frac{\pi}{n}\\
  -2-\frac{1}{\tan^2 \frac{\pi}{n}}-\tan ^2 \frac{\pi}{n} & \frac{2}{\tan ^2 \frac{\pi}{n}}-2\tan^2 \frac{\pi}{n} &  -2-\frac{1}{\tan^2 \frac{\pi}{n}}-\tan ^2 \frac{\pi}{n} & &\cdots & \\
  & -2-\frac{1}{\tan^2 \frac{\pi}{n}}-\tan ^2 \frac{\pi}{n} & \frac{2}{\tan ^2 \frac{\pi}{n}}-2\tan^2 \frac{\pi}{n} & -2-\frac{1}{\tan^2 \frac{\pi}{n}}-\tan ^2 \frac{\pi}{n} & \ddots & \vdots\\
  \vdots & & \ddots & \ddots & \ddots & \\
   &\vdots& \ddots & -2-\frac{1}{\tan^2 \frac{\pi}{n}}-\tan ^2 \frac{\pi}{n}& \frac{2}{\tan ^2 \frac{\pi}{n}}-2\tan^2 \frac{\pi}{n}& -2-\frac{1}{\tan^2 \frac{\pi}{n}}-\tan ^2 \frac{\pi}{n} \\
  -2-\frac{1}{\tan^2 \frac{\pi}{n}}-\tan ^2 \frac{\pi}{n} & & \cdots &  & -2-\frac{1}{\tan^2 \frac{\pi}{n}}-\tan ^2 \frac{\pi}{n} & \frac{2}{\tan ^2 \frac{\pi}{n}}-2\tan^2 \frac{\pi}{n}\\
 \end{smallmatrix}\right)_{n \times n}.
 $$
Furthermore, note that
 \begin{equation}\label{eq: C-B^T B in terms of H}
   K-B^T B = \frac{1}{\cos^2 \frac{\pi}{n}} H 
 \end{equation}
 and
 \begin{equation}\label{eq: D-B^T B in terms of C-B^T B and E}
   D-B^T B = \frac{1}{\sin^2 \frac{\pi}{n} \cos ^2 \frac{\pi}{n}} H,
 \end{equation}
 where
 \begin{equation}\label{eq: H}
H=\begin{pmatrix}
  2\cos \frac{2\pi}{n} &  -1 & 0 & \cdots &0 & -1\\
  -1 & 2\cos \frac{2\pi}{n} &  -1 & 0&\cdots & 0\\
  0& -1 & 2\cos \frac{2\pi}{n} & -1 & \ddots & \vdots\\
  \vdots &0 & \ddots & \ddots & \ddots & 0\\
  0 &\vdots& \ddots & -1& 2\cos \frac{2\pi}{n}& -1\\
  -1 & 0 & \cdots & 0 & -1 & 2\cos \frac{2\pi}{n}\\
 \end{pmatrix} _{n \times n}.
 \end{equation}
Let $U:=Q^T M Q,\ V:=Q^T N Q$. In order to prove \eqref{eq: QM<QN}, we will need to compute the eigenvalues and eigenvectors of $U$ and $V$. This will be achieved via circulant matrix theory.\\ 

\noindent \textbf{Eigenvalues and eigenvectors of $U$.}

\noindent Since $U=\begin{pmatrix}
  I & 0 \\
  0 & K-B^T B
 \end{pmatrix}$, the eigenvalues of $U$ are the eigenvalues of $I$ and the eigenvalues of $K-B^T B$. The only eigenvalue of $I$ is $1$, and the matrix $K-B^T B$ is circulant; hence, we utilize \eqref{eq: eigs_of_circ_matr} to deduce that the eigenvalues of $K-B^T B$, say $\lambda_k$, are given by
 \begin{equation}\label{eq: eig(C-B^T B)}
   \lambda _k= \frac{2\cos \frac{2\pi}{n}}{\cos ^2 \frac{\pi}{n}}-\frac{1}{\cos ^2 \frac{\pi}{n}}(\omega _k+\omega _k^{n-1})=\frac{4\sin \frac{\pi(k-1)}{n}\sin \frac{\pi(k+1)}{n}}{\cos^2 \frac{\pi}{n}},
 \end{equation}
for $k=0,1,\ldots,n-1$. Next, denote the standard basis in $\RR^n$ by $e_1,\ldots, e_n$, and let $v_0,\ldots,v_{n-1} \in \RR ^{n}$ be the vectors given by Proposition \ref{prp: eigs_real_sym_circ_matr}. Recall that $v_k$ is an eigenvector corresponding to the eigenvalue $\lambda _{\lceil \frac{k}{2} \rceil}$ and define $f_k:=(e_k;0,\ldots,0) \in \RR ^{2n}$, for $k=1,2,\ldots, n$ and $f_k=(0,\ldots,0;v_{k-n-1}) \in \RR ^{2n}$, for $k=n+1,\ldots,2n$. Evidently, the vectors $f_i$ form an orthogonal basis in $\RR ^{2n}$ and are eigenvectors of the matrix $U$. Note that the vectors $f_1,f_2,\ldots,f_n$ are the eigenvectors corresponding to the eigenvalue $1$ and the eigenvector $f_k$ corresponds to $\lambda_{\lceil \frac{k-n-1}{2} \rceil}$, for $k=n+1,n+2,\ldots, 2n$.
Now pick any $(\zeta;\eta) \in \RR ^{2n}$. Then there exist unique coefficients $\alpha _i \in \RR$ such that
\begin{equation}\label{eq: expansion in v-s}
  (\zeta;\eta)=\sum \limits _{k=1}^{2n} \alpha _k f_k.
\end{equation}

\noindent Thus, 
\begin{align}\label{eq: Uform in v-s}
  \langle U(\zeta;\eta),(\zeta;\eta) \rangle &= \sum \limits_{k,k'=1}^{2n} \alpha _k \alpha _{k'} \langle U f_k, f_{k'} \rangle=\sum \limits_{k=1}^{n} \alpha _k^2 |f_k|^2+\sum \limits_{k=n+1}^{2n} \alpha _k^2 \lambda _{\lceil \frac{k-n-1}{2} \rceil} |f_k|^2\\ \nonumber
  &=\sum \limits_{k=1}^{n} \alpha _k^2+n \alpha_{n+1}^2 \lambda _0+\sum \limits_{k=n+2}^{2n} \alpha _k^2 \lambda _{\lceil \frac{k-n-1}{2} \rceil}|f_k|^2,
\end{align}
and since $f_{n+1}=(0,0,\ldots,0;1,1,\ldots,1)$, it follows that 
\begin{equation}\label{eq: alpha n+1}
  \alpha _{n+1}=\frac{\langle(\zeta;\eta),f_{n+1}\rangle}{|f_{n+1}|^2}=\frac{\sum \limits _{i=1}^{n} \eta _i}{n}.
\end{equation}
Now, if $(\zeta;\eta) \in \HH _1$, then $\sum \limits _{i=1}^{n} \zeta _i=\sum \limits_{i=1}^{n} \eta _i=0$; therefore, $\alpha _{n+1}=0$. Hence,
\begin{align}\label{eq: Uform in v-s final}
  \langle U(\zeta;\eta),(\zeta;\eta) \rangle =\sum \limits_{k=1}^{n} \alpha _k^2 +\sum \limits_{k=n+2}^{2n} \alpha _k^2 \lambda _{\lceil \frac{k-n-1}{2} \rceil}|f_k|^2.
\end{align}

\noindent \textbf{Eigenvalues and eigenvectors of $V$.}

\noindent Since $V$ has exactly the same form as $U$, our analysis above is valid for $V$. Thus, if $(\zeta;\eta)=\sum \limits _{k=1}^{2n} \alpha _k f_k$, then
\begin{align}\label{eq: Vform in v-s}
  \langle V(\zeta;\eta),(\zeta;\eta) \rangle &= \sum \limits_{k,k'=1}^{2n} \alpha _k \alpha _{k'} \langle V v_k, v_{k'} \rangle=\sum \limits_{k=1}^{n} \alpha _k^2 |f_k|^2+\sum \limits_{k=n+1}^{2n} \alpha _k^2 \mu _{\lceil \frac{k-n-1}{2} \rceil} |f_k|^2\\ \nonumber
  &=\sum \limits_{k=1}^{n} \alpha _k^2+n \alpha_{n+1}^2 \mu _0+\sum \limits_{k=n+2}^{2n} \alpha _k^2 \mu _{\lceil \frac{k-n-1}{2} \rceil}|f_k|^2,
\end{align}
where $\mu _0,\mu _1,\ldots,\mu _{n-1}$ are the eigenvalues of $D-B^T B$, and as before, $\alpha _{n+1}=0$. Next, since $D-B^T B=\frac{1}{\sin ^2 \frac{\pi}{n}}(K-B^T B)$, it follows that
\begin{align}\label{eq: mu-s in lambda-s }
   \mu _k=\frac{\lambda _k}{\sin ^2 \frac{\pi}{n}},
 \end{align}
for $k=0,1,\ldots,n-1$.
Hence,
\begin{align}\label{eq: Vform in v-s final}
  \langle V(\zeta;\eta),(\zeta;\eta) \rangle =\sum \limits_{k=1}^{n} \alpha _k^2 +\frac{1}{\sin ^2 \frac{\pi}{n}}\sum \limits_{k=n+2}^{2n} \alpha _k^2 \lambda _{\lceil \frac{k-n-1}{2} \rceil}|f_k|^2,
\end{align}
which together with \eqref{eq: Uform in v-s final} and the fact that $\lambda _k \geq 0$ for $1\leq k \leq n-1$, implies
\begin{equation*}
  \langle U(\zeta;\eta),(\zeta;\eta) \rangle \leq \langle V(\zeta;\eta),(\zeta;\eta) \rangle.
\end{equation*}

\noindent Thus, condition (iii) of Lemma \ref{cond} holds. It remains to verify (iv), which requires
$$\langle D^2f(z_*) w,w \rangle >0,$$ for $w \in \HH$ and $w\neq 0.$ Since $D^2f(z_*)=F=\cos \frac{\pi}{n} M$, this is equivalent to
$$\langle M w,w \rangle >0,$$ for $w \in \HH$ and $w\neq 0$, which is in turn equivalent to
\begin{itemize}
  \item[(iv)'] $\langle U (\zeta;\eta),(\zeta;\eta) \rangle >0$, for $(\zeta;\eta) \in \tilde{\HH}$ and $(\zeta;\eta)\neq 0$,
\end{itemize}
where $\tilde{\HH}$ is the space that the space $\HH$ is mapped to under the coordinate transformation \eqref{eq: coordtransform}. In order to identify $\tilde{\HH}$, note that since $x_k=\zeta _k-\tan \frac{\pi}{n}(\eta _k+\eta _{k+1})$ and $r_k=\eta _k$, 
\begin{align}
  \sum \limits _{k=1}^{n} x_k=0 &\quad \mbox{iff} \quad \sum \limits _{k=1}^{n} \zeta _k=2 \tan \frac{\pi}{n} \sum \limits _{k=1}^{n} \eta_k,\\ \nonumber
  \sum \limits _{k=1}^{n} r_k=0 &\quad \mbox{iff} \quad \sum \limits _{k=1}^{n} \eta _k=0.
\end{align}
Furthermore, the condition
\begin{equation*}
  -\frac{1}{2 \sin \frac{\pi}{n}}\sum \limits _{k=1}^{n} x_k \cos \frac{\pi(2k-1)}{n}+\sum \limits _{k=1}^{n} r_k \cos \frac{2 \pi (k-1)}{n}=0
\end{equation*}
transforms into
\begin{equation*}
  -\frac{1}{2 \sin \frac{\pi}{n}}\sum \limits _{k=1}^{n} \left(\zeta _k-\tan \frac{\pi}{n}(\eta_k+\eta _{k+1})\right) \cos \frac{\pi(2k-1)}{n}+\sum \limits _{k=1}^{n} \eta_k \cos \frac{2 \pi (k-1)}{n}=0,
\end{equation*}
which after simplification becomes
\begin{equation}
  -\frac{1}{2 \sin \frac{\pi}{n}}\sum \limits _{k=1}^{n} \zeta _k \cos \frac{\pi(2k-1)}{n}+2\sum \limits _{k=1}^{n} \eta_k \cos \frac{2 \pi (k-1)}{n}=0.
\end{equation}
Similarly, the condition
\begin{equation*}
  -\frac{1}{2 \sin \frac{\pi}{n}}\sum \limits _{k=1}^{n} x_k \sin \frac{\pi(2k-1)}{n}+\sum \limits _{k=1}^{n} r_k \sin \frac{2 \pi (k-1)}{n}=0
\end{equation*}
transforms into
\begin{equation}
  -\frac{1}{2 \sin \frac{\pi}{n}}\sum \limits _{k=1}^{n} \zeta _k \sin \frac{\pi(2k-1)}{n}+2\sum \limits _{k=1}^{n} \eta_k \sin \frac{2 \pi (k-1)}{n}=0.
\end{equation}
Therefore, the space $\tilde{\HH}$ is given by
\begin{align*}
  \sum \limits_{k=1}^{n} \zeta_k&=2 \tan \frac{\pi}{n}\sum \limits_{k=1}^{n} \eta_k;\\
  \sum \limits_{k=1}^{n} \eta_k&=0;\\
  -\frac{1}{2 \sin \frac{\pi}{n}}\sum \limits _{k=1}^{n} \zeta _k \cos \frac{\pi(2k-1)}{n}+2\sum \limits _{k=1}^{n} \eta_k \cos \frac{2 \pi (k-1)}{n}&=0;\\
  -\frac{1}{2 \sin \frac{\pi}{n}}\sum \limits _{k=1}^{n} \zeta _k \sin \frac{\pi(2k-1)}{n}+2\sum \limits _{k=1}^{n} \eta_k \sin \frac{2 \pi (k-1)}{n}&=0;
\end{align*}
or equivalently by
\begin{align}\label{eq: tilde H}
  \sum \limits_{k=1}^{n} \zeta_k&=0;\\ \nonumber
  \sum \limits_{k=1}^{n} \eta_k&=0;\\ \nonumber
  -\frac{1}{2 \sin \frac{\pi}{n}}\sum \limits _{k=1}^{n} \zeta _k \cos \frac{\pi(2k-1)}{n}+2\sum \limits _{k=1}^{n} \eta_k \cos \frac{2 \pi (k-1)}{n}&=0;\\ \nonumber
  -\frac{1}{2 \sin \frac{\pi}{n}}\sum \limits _{k=1}^{n} \zeta _k \sin \frac{\pi(2k-1)}{n}+2\sum \limits _{k=1}^{n} \eta_k \sin \frac{2 \pi (k-1)}{n}&=0.
\end{align}
Note that \eqref{eq: tilde H} is simply the condition that the vector $(\zeta;\eta)$ is orthogonal to the vectors 
\begin{align*}
&w_1=\sum \limits _{k=1}^{n} f_k,\\
&w_2 =f_{n+1},\\
&w_3=-\frac{1}{2 \sin \frac{\pi}{n}}\sum \limits _{k=1}^{n} f_k \cos \frac{\pi(2k-1)}{n}+2 f_{n+2},\\
&w_4=-\frac{1}{2 \sin \frac{\pi}{n}}\sum \limits _{k=1}^{n} f_k \sin \frac{\pi(2k-1)}{n}+2 f_{n+3}.
\end{align*}
Moreover,  
\begin{align*}
  \langle (\zeta;\eta),w_1\rangle &=\sum \limits_{k=1}^{n} \alpha_k,\\
  \langle (\zeta;\eta),w_2\rangle &=n\alpha_{n+1},\\
  \langle (\zeta;\eta),w_3\rangle &=n\alpha_{n+2}-\frac{1}{2 \sin \frac{\pi}{n}}\sum \limits_{k=1}^{n} \alpha_k \cos \frac{\pi(2k-1)}{n},\\
  \langle (\zeta;\eta),w_3\rangle &=n\alpha_{n+3}-\frac{1}{2 \sin \frac{\pi}{n}}\sum \limits_{k=1}^{n} \alpha_k \sin \frac{\pi(2k-1)}{n},
\end{align*}
where we utilized the fact that $|f_{n+2}|^2, |f_{n+3}|^2=\frac{n}{2}$. Thus, we see that $(\zeta;\eta) \in \tilde{\HH}$ is identified by the equations 
\begin{align*}
  \sum \limits_{k=1}^{n} \alpha_k&=0,\\
  \alpha_{n+1}&=0,\\
  \alpha_{n+2}&=\frac{1}{2n \sin \frac{\pi}{n}}\sum \limits_{k=1}^{n} \alpha_k \cos \frac{\pi(2k-1)}{n},\\
  \alpha_{n+3}&=\frac{1}{2n \sin \frac{\pi}{n}}\sum \limits_{k=1}^{n} \alpha_k \sin \frac{\pi(2k-1)}{n},
\end{align*}
and an application of Cauchy-Schwarz yields 
\begin{align*}
  \alpha_{n+2}^2 \leq \frac{\sum \limits_{k=1}^{n}  \cos ^2 \frac{\pi(2k-1)}{n}}{4n^2 \sin^2 \frac{\pi}{n}} \sum \limits_{k=1}^{n} \alpha_k^2,\\
  \alpha_{n+3}^2 \leq \frac{\sum \limits_{k=1}^{n}  \sin ^2 \frac{\pi(2k-1)}{n}}{4n^2 \sin^2 \frac{\pi}{n}} \sum \limits_{k=1}^{n} \alpha_k^2;
\end{align*}
hence,
\begin{align}
  \sum \limits_{k=1}^{n+3} \alpha _k^2 \leq \Big(1+\frac{1}{4n \sin^2 \frac{\pi}{n}} \Big) \sum \limits_{k=1}^{n} \alpha _k^2,
\end{align}
and by letting $\tilde c_n:=\Big(1+\frac{1}{4n \sin^2 \frac{\pi}{n}} \Big)^{-1}$, it follows that 
\begin{align*}
  \langle U(\zeta;\eta),(\zeta;\eta)\rangle &\ge\sum \limits_{k=1}^{n} \alpha _k^2+\sum \limits_{k=n+4}^{2n} \alpha _k^2 \lambda_{\lceil \frac{k-n-1}{2}\rceil} |f_k|^2 \\
  &\geq \tilde c_n \sum \limits_{k=1}^{n+3} \alpha _k^2+\sum \limits_{k=n+4}^{2n} \alpha _k^2 \lambda_{\lceil \frac{k-n-1}{2} \rceil} |f_k|^2\\
  &\geq c_n \sum \limits_{k=1}^{2n} \alpha _k^2,
\end{align*}
where $c_n=\min \{\tilde c_n,\{\lambda_{\lceil \frac{k-n-1}{2}\rceil} |f_k|^2\}_{k=n+4}^{2n}\}>0$. Therefore, $\langle U(\zeta;\eta),(\zeta;\eta)\rangle >0$, whenever $(\zeta;\eta) \in \tilde{\HH}$ and $(\zeta;\eta)\neq 0$. Thus, condition (iv) is verified and we conclude the proof the theorem for convex $n$-gons.

\begin{remark} \label{sharp}
Since we showed that $f$ and $g$ and their gradients vanish at the minimizer and the Hessian of $f$ is non-zero and bounds the Hessian of $g$ from below (at the minimzer), it follows that the estimate in Theorem \ref{thm: main} is sharp in the exponents.  
\end{remark}

\begin{remark} \label{genth}
We note that any function $f$ satisfying the conditions of Lemma \ref{cond} serves as a lower bound on $g$, which in our functional formulation represents the polygonal isoperimetric deficit. In particular, our method may be useful in obtaining new geometrically significant lower bounds on the deficit.  
\end{remark}

\subsection{The general case} \label{gen}

In this section we reduce the problem to the convex case by utilizing the generalization of the Erd{\H o}s-Nagy theorem to non-simple polygons given by Toussaint \cite{EN}. Recall from the introduction that $$\tau(P):=\sum_{i=1}^{k_n} \tau_i(P),$$ where $\tau_i(P)$ is the area increase at the $i$-th step given by Toussaint's process of selecting flips. Since the perimeter is invariant at each step, we have 
\begin{align*}
\delta(P)&=L^2(P)-c_n|P|\\ 
&=L^2(P_1)-c_n(|P_1|-\tau_1(P))\\
&=\delta(P_1)+c_n \tau_1(P)\\
&=\cdots\\
&=\delta(P_{k_n})+c_n \tau(P),
\end{align*}

\noindent where $P_{k_n}$ is a convex $n$-gon; thus, by what we proved in the previous sections it follows that $$\sigma_s^2(P_{k_n}) \le C(n) \delta(P_{k_n}).$$ But $\sigma_s^2(P_{k_n})=\sigma_s^2(P)$ since the flipping process preserves the lengths of the sides; combining this with \eqref{fin} yields
\begin{align*} 
8n^2  \sin ^2 \frac{\pi}{n} \ \sigma_r ^2(P) &\le \delta(P)+n^2 \sigma_s ^2(P)\\
&\le \delta(P)+n^2 C(n) (\delta(P)-c_n\tau(P)).    
\end{align*} 
Thus, $$\tau(P)+v(P) \lesssim \delta(P),$$ and this finishes the proof for non-simple polygons. 

\noindent The proof of Corolary \ref{thm: main2} follows from the observation that if the polygons are simple, then the interior of $P_1$ contains the interior of $P$ and so $$|P_1\Delta P|=|P_1\setminus P|=|P_1|-|P|;$$ thus, the total area increase is given by $$|P \Delta P_1|+|P_1 \Delta P_2|+\cdots+|P_{k_n-1} \Delta P_{k_n}|,$$ and the triangle inequality in $L^1$ implies the result (of course, we take $P_c:=P_{k_n}$ in this case). 
  
%
%
%
 
\bibliographystyle{alpha}

\bibliography{ngonref}

\signei

\signln

\end{document}